\documentclass[12pt]{article}

\usepackage[backend=bibtex, doi=false, url=false, isbn=false]{biblatex}
\DeclareNameAlias{sortname}{last-first}

\usepackage{fullpage}
\usepackage[T1]{fontenc}
\usepackage{lmodern}
\usepackage{hyperref}
\usepackage{xcolor}
\definecolor{link-color}{rgb}{0.15,0.4,0.15}
\hypersetup{
  colorlinks,
  linkcolor = {link-color}, citecolor = {link-color},
}

\usepackage{graphicx}
\usepackage[font={small}]{caption}
\usepackage{subcaption}
\usepackage{hyperref}
\usepackage{amsmath,amsthm,amssymb}
\usepackage{bbm}
\usepackage{tabularx}

\newtheorem{theorem}{Theorem}[section]

\newtheorem{lemma}[theorem]{Lemma}
\newtheorem{conjecture}[theorem]{Conjecture}
\theoremstyle{plain}

\newcommand{\N}{\mathbb{N}}
\renewcommand{\P}{\mathbb{P}}
\newcommand{\E}{\mathbb{E}}
\newcommand{\1}{\mathbbm{1}}

\newenvironment{enum}{\begin{list}{(\roman{enumi})}{\usecounter{enumi}}}{\end{list}}

\usepackage{bm}

\newcommand{\orb}{{\rm orb}}

\title{Existence of a phase transition of the interchange process on the Hamming graph}
\author{Piotr Mi\l{}o\'s\footnote{University of Warsaw, MIMUW, Banacha 2, 02--097 Warszawa, Poland, Email: \texttt{pmilos@mimuw.edu.pl}}, Bat\i{} \c{S}eng\"ul\footnote{Department of Mathematical Sciences, University of Bath, Claverton Down, Bath, BA2 7AY, UK.\@ Email: \texttt{batisengul@gmail.com}}}

\bibliography{references}

\begin{document}

\maketitle

\begin{abstract}
	The interchange process on a finite graph is obtained by placing a particle on each vertex of the graph, then at rate $1$, selecting an edge uniformly at random and swapping the two particles at either end of this edge.
	In this paper we develop new techniques to show the existence of a phase transition of the interchange process on the $2$-dimensional Hamming graph.
  We show that in the subcritical phase, all of the cycles of the process have length $O(\log n)$, whereas in the supercritical phase a positive density of vertices lie in cycles of length at least $n^{2-\varepsilon}$ for any $\varepsilon>0$.
\end{abstract}

\section{Introduction and main results}
The interchange process $\sigma=(\sigma_t:t \geq 0)$ is defined on a finite, connected, undirected graph $G=(V,E)$ as follows.
For each $t \geq 0$, $\sigma_t:V\rightarrow V$ is a permutation on the vertices of the graph and initially $\sigma_0$ is the identity permutation.
Then at rate $1$, an edge $e \in E$ is selected uniformly at random and the vertices on either end of $e$ are swapped in the permutation. Under the assumption of bounded degree, this can be defined similarly on infinite graphs as well.

This model was introduced by \textcite{MR1224836} who used it to derive results about the quantum Heisenberg ferromagnetic model.
Since then, this has been studied on infinite regular trees~\cite{MR2042369,MR3097418,MR3334273} and the complete graph~\cite{MR2166362,MR2754801,MR3351179}.
Recently, \textcite{hypercube_stirring} extended the techniques of \textcite{MR2754801} to show some results about the interchange process on the hypercube.
Kozma and Sidoravicius have announced that they have computed the expected cycle size on $\mathbb Z$ and also similar questions have been answered for a different, but related model by \textcite{2016arXiv160106991G}.

The interchange process is of interest in both physics and mathematics, where outside of the mentioned work, very little is known.
The prior results on trees and the complete graph use very particular properties of these graphs which make them rather hard to extend to other graphs.
It is widely believed that similar results should hold for a large variety of graphs.
An important open conjecture is that for the interchange process on $\mathbb{Z}^d$, there exists a time $t_c=t_c(d)$ such that for $t < t_c$, every cycle of $\sigma_t$ has finite length and for $t>t_c$, $\sigma_t$ has a cycle of infinite length. Moreover it is believed that $t_c<\infty$ if and only if $d \geq 3$.

In this paper we consider the $2$--dimensional Hamming graph $H(2,n)=(V,E)$ for which the vertices $V=\{0,\dots,n-1\}^2$ are given by the square lattice and an edge is present between any pair of vertices which either are on the same row or the same column.
We will often denote this graph by $H$.
The Hamming graph is the product of two complete graphs, so it is somewhat surprising that applying the techniques on the complete graph directly gives results that are far from optimal.
Instead, in this paper we obtain results that are almost optimal, for which we have to develop a lot of new ideas. We believe that this work is a step in moving away from the complete graph, which is well understood, towards other graphs with different geometry, and in particular, towards $\mathbb{Z}^d$.

Our main result is the occurrence of a phase transition for the interchange process on $H$.

\begin{theorem}\label{thm:phase}
	Consider the random interchange process $(\sigma_t:t \geq 0)$ on the $2$--dimensional Hamming graph $H(2,n)$.
	For $v\in V$ define the cycle containing $v$ by
	\[
		\orb_t(v)=\{\underbrace{\sigma_t\circ\dots\circ\sigma_t}_{\ell}(v): \ell=0,1,\dots\}.
	\]
	and vertices belonging to cycles longer than $k$
	\[
		V_t(k) = \{v\in V: \orb_t(v)\geq k\}.
	\]

  Let $t=\beta n^2$, then process exhibits a phase transition. For $\beta<1/2$ we are in subcritical phase where there exist a constant $C>0$ such that
		\[
			\lim_{n \rightarrow \infty}\P(|V_t(C \log n)|=0)=1.
		\]
	For $\beta>1/2$ we are in the supercritical phase where for any $\varepsilon>0$ there exists a constant $C>0$ such that
		\[
			\lim_{n \rightarrow \infty}\P\left(\big |V_t(n^{2-\varepsilon})\big |\geq Cn^2 \right) =1.
		\]
\end{theorem}

We expect that our result are almost sharp and conjecture that for $\beta>1/2$ with high probability there exists cycles of macroscopic size.

The main technical contribution of our paper, instrumental to show Theorem \ref{thm:phase}, is to show that cycles cannot concentrate too much on any row or column of the graph $H(2,n)$. Roughly speaking it is expected that the cycles of the interchange process should resemble is many aspects a random walk on the graph. To the best of our knowledge, we are the first to successfully implement this idea for the interchange model.

Refining the estimates in this paper and adapting the argument of~\textcite{MR2754801}, we can show that in supercritical phase
\begin{equation}\label{eq:whatIfWeUsedBerestycki}
	\lim_{n \rightarrow \infty}\P\left( \sup_{u\in[t-\Delta_n, t]} \big|V_u(n^{2}/\log^4 n)\big |\geq Cn^2 \right) =1,
\end{equation}
for any $\Delta_n\geq n^{1/2}\log^3 n$.
We predict that the $\log^4 n$ above could also be replaced by polylog$(n)$ with sharper estimates.
To replace the $\log^4 n$ by a constant, thereby showing the existence of macroscopic cycles, would require some new ideas.


We discuss the above questions in more detail in Section~\ref{sec:openQuestionsAndDiscussion}.

\subsection{Heuristics and the outline of the paper}

The subcritical phase will follow from an easy coupling with percolation on $H$ and so we focus on explaining the heuristics behind the supercritical phase.

In the supercritical phase we proceed iteratively proving existence of larger and larger cycles.
Let $t = \beta n^2$ with $\beta>1/2$.
It is easy to show that for small $\alpha_1>0$ cycles of size $n^{\alpha_1}$ occupy a positive density of vertices.
We introduce a random graph process $G^t=(G^t_s:s \geq 0)$. Initially $G_0^t$ is a graph whose connected components are the cycles of $\sigma_t$. Next, whenever $(\sigma_{t+s}:s \geq 0)$ swaps a pair of particles across an edge $e$, we add $e$ to the graph process.
As we increase $s$, the largest component in $G^t_s$ becomes giant rather quickly; in fact, an easy sprinkling argument shows that this happens after $n^{2-\alpha_1} \log n$ units of time.
In this short time, there cannot be too many splits which result in cycles of size less than $n^{\alpha_2}$ for some $\alpha_2>0$. In other words, almost every vertex lying inside of the giant component of $G^t$ belongs to a cycle bigger than $n^{\alpha_2}$.

Clearly, the bigger $\alpha_2$ the better, and the main difficulty is that there might be cycles which ``split too easily''.
A pathological example is a cycle which has its support on a single line or a single column.
We show is that with high probability most cycles do not behave in this pathological way.
In some aspects the cycle of $\sigma_t$ resembles the trace of a simple random walk on $H$, which in turn resembles a set of i.i.d. uniformly chosen point. In the last case it is not hard to observe that the vertices cannot cluster in any line or column.
Formalising the above is an extremely important technical result and turns out to be rather delicate.

In the proof as an input we use information that cycles of size $n^{\alpha_1}$ are common. Once we obtain cycles of size $n^{\alpha_2}$ we can can repeat the analysis obtaining better estimates and proving existence of even longer cycles of size $n^{\alpha_3}$, for $\alpha_3>\alpha_2$. Our methods are sharp enough to continue inductively with a sequence $\alpha_k$ converging to $2$.

\subsection*{Outline of the paper}
The paper is orginised as follows. In Section~\ref{sec:graph_notation} we introduce our notion of isoperimetry and show some basic results.
In Section~\ref{sec:crw} we introduce the cyclic random walk, which plays an important role in our proof. Next in Section~\ref{sec:isoperimetry} we use the cyclic random walk to give bounds on the isopermetry of the cycles of $\sigma_t$.
In Section~\ref{sec:longCyclesUnderGoodIsoperymetry} we obtain results which give lower bounds for the cycle lengths, under the assumption of good isopermetry.
In Section~\ref{sec:mainProof} we combine the previous results to show Theorem~\ref{thm:phase}. Finally, Section \ref{sec:openQuestionsAndDiscussion} contains open questions and further discussion.

\section{Definitions and preliminary results}
\subsection{Isoperimetry}\label{sec:graph_notation}

Let $H=H(2,n)$ be the $2$--dimensional Hamming graph.
Let $V$ denote the vertices and $E$ denote the edges of $H$.
Recall that the vertices $V=\{0,\dots,n-1\}^2$.
For $i\in\{0,\ldots,n-1\}$ sets $L_i =\{0,\dots,n-1\}\times \{i\}$ will be called rows and $D_i=\{i\}\times \{0,\ldots, n-1\}$ columns.
An edge $e \in E$ is placed between any two distinct vertices on the same row or column.
One can check that $|V|=n^2$ and $|E|=n^2(n-1)$.
In the whole paper we assume implicitly that $n\geq 2$.

In the proofs we will extensively use isoperimetric properties of sets.
Let us fix the notation.
Given $A,B\subset V$ by $E(A,B)$ denote the set of edges $(v,w)\in E$ such that $v \in A$ and $w \in B$.
Next we define a notion of isoperimetry for a set $A\subset V$ by setting
\begin{equation}\label{eq:iotaDefinition}
	\iota(A) :=\max\left\{\max_{i\in \{0,\dots,n-1\}} |L_i\cap A|\, ,\, \max_{i\in \{0,\dots,n-1\}} |D_i\cap A| \right\}.
\end{equation}
We think of $\iota(A)$ as an isopermetric constant of $A$.
Notice that $\iota$ is sub-additive in the sense that for $A, B \subset V$, $\iota(A \cup B)\leq \iota(A) + \iota(B)$.

The next lemma shows how to bound $|E(A,A)|$ using $\iota(A)$.

\begin{lemma}\label{lem:isoperymetryByIota}
	Let $A\subset V$ then
	\[
		|E(A,A)|\leq |A|\iota(A).
	\]
\end{lemma}
\begin{proof}
	Fix $A \subset V$.
	For each $v \in A$, $v$ has at most $2 \iota(A)$ many neighbours, thus the total number of edges from $v$ to vertices of $A$ is at most $(2\iota(A)|A|)/2$, where the division by $2$ comes from the fact that each edge is counted twice.
\end{proof}

Given $v,w\in V$ we set $v+w$ to be component-wise addition modulo $n-1$.
Further, for $A\subset V$ and $v\in V$ we put $v+A=\{v+w:w\in A\}$.
The next lemma follows easily from the definition and thus we leave the proof out.

\begin{lemma}\label{lem:combinatorial}
	Let $A,B \subset V$ then
	\[
		\sum_{v\in V} |(v+A)\cap B| = |A||B|, \quad \sum_{v\in L_0} |(v+A)\cap B| =  \sum_{i=0}^{n-1} |A\cap L_i||B\cap L_i|.
	\]
\end{lemma}

\subsection{Poissonian construction and the cyclic random walk}\label{sec:crw}
We think of the interchange process as a Poisson point process on the edges $E$ of $H$ and construct it as follows.
Consider a Poisson point process $\mathcal M$ on $E\times [0,\infty)$ with intensity measure given by $|E|^{-1}\#(\cdot)\otimes{\rm Leb}$ where $\#(\cdot)$ is the counting measure and ${\rm Leb}$ is the Lebesgue measure.
For each $t \geq 0$ we set
\[
	\mathcal B_t:=\{(e,z)\in \mathcal M: z \leq t\}
\]
to be the restriction of $\mathcal M$ to $E \times [0,t]$.

We call each $b \in \mathcal B_t$ a \emph{bridge} and call $\{v\}\times [0,t]$ the \emph{bar} at vertex $v$.
We think of a bridge $((v,w),z)$ as going across two bars from vertex $v$ to vertex $w$ at time $z\in [0,t]$.
We let $\mathcal B_t(A,B)$ be the set of bridges $b=(e,t)\in\mathcal B$ such that $e\in E(A,B)$.
In cases when $t$ is fixed we will often drop it from the notation by write $\mathcal B=\mathcal B_{t}$.

Let $v \in V$, then we obtain $\sigma_t(v)$ from the following procedure (see Figure~\ref{fig:crw}).
We start at $(v,0)\in H\times [0,t]$ and follow the bar $\{v\}\times [0,t]$ until we reach the first bridge $(e,z)\in\mathcal B_t$, if it exists, such that $e=(v,w)$ for some $w \in V$.
Then we jump to $(w,z)$ and then again follow the interval $\{w\}\times [z,t]$ until the next bridge.
Repeating this procedure, we eventually end up at some $(v',t)\in H\times [0,t]$ and we have that $\sigma_t(v)=v'$.

\begin{figure}
	\centering
	\begin{subfigure}[b]{0.45\textwidth}
		\includegraphics[width=\linewidth]{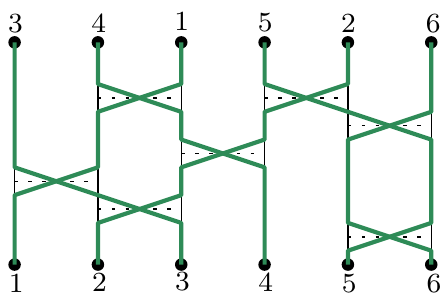}
	\end{subfigure}
	\hfill
	\begin{subfigure}[b]{0.45\textwidth}
		\includegraphics[width=\linewidth]{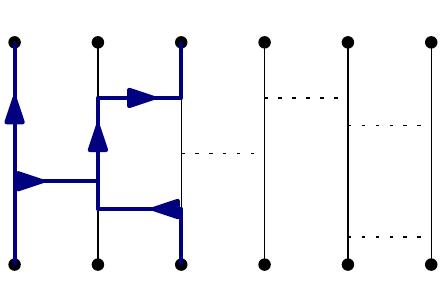}
	\end{subfigure}
	\caption{The dotted lines represent the bridges $\mathcal B$.
		The picture on the left shows how to obtain $\sigma_t$, the labels at the bottom are the labels of the vertices and at the top we have put where they map to under $\sigma_t$.
		The figure on the right is the path of the CRW which is in blue and the direction that the CRW travels is indicated by the arrows.}\label{fig:crw}
\end{figure}

Now we present the \emph{cyclic random walk} (CRW), introduced by~\textcite{MR2042369}, which will explore the set of bridges $\mathcal B$ in a convenient way.
The CRW $\mathcal X=(\mathcal X_s:s\geq 0)$ is a continuous time process which takes values in $V\times [0,t]$ and is defined as follows (see Figure~\ref{fig:crw}).
Initially we start at a point $\mathcal X_0=(v,z) \in V\times [0,t]$.
The CRW moves upwards on the bar of the vertex $v$, starting at height $z$, at unit speed, until it encounters a bridge.
It does this periodically, so that if it gets to height $t$, then the CRW goes to the bottom of the bar, at height $0$.
If a bridge has been encountered, then the CRW jumps to the other end of the bridge and repeats the same procedure.
We will be assuming that $\mathcal X$ is a right continuous function.

Notice that the CRW is periodic.
For example in the case there are no bridges coming out of $v$, $\mathcal X$ will be given by $\mathcal X_s=(v,z+s \mod t)$.
In general, once $\mathcal X$ reaches the point $(v,z)$ again (which it will do so in a finite time), then it will repeat itself.
Crucially, the paths of $\mathcal X$ encode the permutation $\sigma_t$.
Indeed note that the vertices $\orb_t(v)$ in the cycle of $v$ can be obtained from $\mathcal X$ as follows:
\[
	\orb_t(v)=\{w: \mathcal X_{s}=(w,t) \, \text{ for some }s \geq 0\}.
\]
We are interested in the trace of $\mathcal X$ which we define as
\[
	\mathcal Z_s:=\{v\in V: \mathcal X_{s'}=(v,z)\text{ for some }s'\leq s\text{ and }z\in [0,t]\}.
\]

For $x>0$ we let $T_x:=\inf\{s \geq 0: |\mathcal Z_s|\geq x\}$ where we use the convention that $\inf \emptyset=\infty$.
We define $(Z_k: k=0,1,\dots)$ and  $(X_k: k = 0,1,\dots)$ by $Z_k:=\mathcal Z_{T_k}$ and
\[
	X_k:=\begin{cases}
	Z_1 & \text{ if }k=1\\
	Z_{k}\backslash Z_{k-1} & \text{ if } k \geq 2.
	\end{cases}
\]
Note that $X_k = \emptyset$ if $T_k = \infty$.
When $X_k\neq \emptyset$ we will often ignore the fact that $X_k$ is a set and write $X_k=v$ instead of $X_k=\{v\}$.
In some cases $t$ or the starting position of $(v,z)$ will be important and we denote this by $\mathcal Z_s(v,z), T_k(v,z)$ etc.

Let $\mathcal F=(\mathcal F_s:s \geq 0)$ denote the natural filtration of the CRW $\mathcal X$.
We also set  $\mathcal G_k:=\mathcal F_{T_k}$.
It is important to note that $\mathcal G_k$ is finer than the natural filtration of $X$ as it records all the bridges that $\mathcal X$ has crossed, prior to and including time $T_k$.

The four processes $\mathcal X$, $\mathcal Z$, $X$ and $Z$ are all measurable with respect to the set of bridges $\mathcal B$.
This means that we only have one source of randomness, nevertheless these processes are a useful tool to explore the set of bridges $\mathcal B$.

The homogeneity of the Poisson point process gives that
\[
	(\mathcal{X}_s(v,z),s\geq 0) \overset{d}{=}  ((v,z)+\mathcal{X}_s(0,0), s\geq 0)
\]
where $+$ is applied component-wise.
In particular this implies similar equalities in distribution for $\mathcal{X}, \mathcal{Z}, X,Z$ and $T$, for example
\[
	(Z_k(v,z), k=0,1\dots) \overset{d}{=}(v+Z_k(0,0), k=0,1,\dots).
\]

\section{Isoperimetric properties of cycles}\label{sec:isoperimetry}
In this section we fix $t>0$ and in most cases we drop it from the notation.

The goal of this section is to show that if $\liminf_{n \to\infty}\P(T_k<\infty)>0$, then $Z_k$ has good isoperimetry, where we think of good isopermitery as saying that $\iota(Z_k)\leq \log^2 n$ with high probability.
We first give the heuristics of this section.

\subsection*{Heuristics of the section}

Imagine first that after $\ell$ steps the CRW is restarted from a point $v$ chosen uniformly from~$V$.
Assuming that $\ell \ll n$ most of the graph has not been explored thus after the restart, then if $T_{k-\ell}(v,0)<\infty$, the next $k-\ell\ll n$ steps of the restarted CRW looks like an independent CRW started from a uniformly chosen vertex.
The CRW started from a uniformly chosen point is very unlikely to intersect any given line, say $L_0$, thus we see that  on $T_{k-\ell}(v,0)<\infty$ after the restart, the original CRW will not hit $L_0$ again with high probability.

Next step is deciding after which event does the CRW end up in a roughly uniform vertex.
We choose it to be an L shaped jump, i.e.\ $X_{\ell}$ to $X_{\ell+1}$ is a horizontal jump and $X_{\ell+1}$ to $X_{\ell+2}$ is a vertical jump.
We observe that each time the CRW visits $L_0$ it has a positive chance of jumping our via an L shaped jump and further doing an excursion of length $k-\ell$ without touching $L_0$.
This leads to the conclusion, stated in Lemma~\ref{lemma:intersect_line}, that the tail of $|Z_k \cap L_0|$ has exponential decay.

This result is strong enough to allow use of a union bound to show that with a very small probability the intersection of $Z_k$ with any line or column is small.
Finally in Lemma~\ref{lem:ispoerymetryAtGivenTime} we will transfer this result to cycles of the permutation $\sigma_t$ and in Lemma~\ref{lemma:T_k_to_orb} we restate them in a form applicable in further sections.

\bigskip

We begin with the central lemma to our argument.
Instead of conditioning on the L shaped jump mentioned above, we let $\mathcal L_{\ell+2}$ denote a suitable event which we will later fix to be the L shaped jump.

\begin{lemma}\label{lemma:one_disjoint_line}
	Let $\ell,k \in \N$ with $\ell+2\leq k$,  and let $\mathcal L_{\ell+2}$ be an event which is $\mathcal G_{\ell+2}$ measurable such that
  \[
    \{X_{\ell+2}\notin L_0\}\cap \{T_{\ell+2}<\infty\} \subset \mathcal L_{\ell+2}.
  \]
  Then
	\begin{align*}
		\P(|(Z_{k} & \backslash Z_{\ell})\cap L_0|\leq 1 | \mathcal G_{\ell})                                                                                                                \\
		           & \geq \left(\P(T_k < \infty) - k(\ell+n) \max_{v\in V}\P(X_{\ell+2}=v|\mathcal L_{\ell+2};\mathcal G_\ell) -\frac{k}{n-1}\right)\P(\mathcal L_{\ell+2}|\mathcal G_\ell).
	\end{align*}
\end{lemma}
\begin{proof}
	Until further mention, we work conditionally on $\mathcal G_{\ell+2}$.
	Fix $D \subset V$ and $(v,z) \in D^c\times [0,t]$.
	We additionally condition on $\{Z_{\ell+1}=D;\mathcal X_{T_{\ell+2}}=(v,z); T_{\ell+2}<\infty\}$.
	We can assume that $\P(\mathcal L_{\ell+2}|\mathcal G_\ell)>0$ then conditioning is well-defined (note that if $\P(\mathcal L_{\ell+2}|\mathcal G_\ell)=0$ there is nothing to prove).

	Let $A=(D \cup L_0)^c$.
	Let us analyse the set of bridges $\mathcal{B}(A,A^c)$.
	Under our conditioning $\mathcal{B}(A,A^c)$ is a thinned Poisson point process with one additional bridge $b$, the one ending at $(v,z)$.
	Indeed, the path of $\mathcal X$ until time $T_{\ell+1}$ excludes certain bridges to be present in $\mathcal B(A,A^c)$.
	Let us resample missing bridges, remove $b$ and finally denote the result by $\tilde{\mathcal B}$.
  The resulting set of bridge $\tilde{\mathcal B}$ is a Poisson point process of bridges which has the same law as the unconditional law of $\mathcal B$.
  Moreover $\tilde{\mathcal B}(A,A)=\mathcal B(A,A)$ which is independent of $\mathcal G_{\ell+2}$. We denote the CRW using this set of bridges and starting from $(v,z)$ by $\tilde {\mathcal X}$.
	Likewise, other quantities associated with this CRW will be decorated with a tilde.

	Let
	\[
		\mathcal{E}=\mathcal E(D,(v,z)):=\{ \tilde{T}_k<\infty; \tilde{Z}_k \subset A\}.
	\]
	Then we claim that on the event $\mathcal E$, the paths $(\tilde{\mathcal{X}}_s: s \leq \tilde T_k)$ and $(\mathcal{X}_{T_{\ell+2}+s}: s \leq \tilde T_k)$ agree.
	Indeed, suppose for a contradiction that this is not true and let
	\[
		\tau = \inf\{u \geq 0: \tilde{\mathcal {X}}_u \neq \mathcal X_{T_{\ell+2}+u}\}.
	\]
	be the first time they disagree so that $\tau<\tilde T_k$.
	Let $(w,z_0)=\tilde{\mathcal X}_{\tau-}$.
	Clearly $w\in A$.
	Suppose first that $(w,z_0)\neq (v,z)$, then there exists a bridge from $(w,z_0)$ to $A^c$ in $\tilde{\mathcal{B}}$.
	This however contradicts the event $\{\tilde{Z}_k \subset A\}$.
	Suppose now that $(w,z_0)=(v,z)$, then the CRW $\tilde{\mathcal X}$ closes into a cycle and behaves periodically after $\tau$, contradicting the event $\{\tilde{T}_k<\infty\}$.

	Since the paths $(\tilde{\mathcal{X}}_s: s \leq \tilde T_k)$ and $(\mathcal{X}_{T_{\ell+2}+s}: s \leq \tilde T_k)$ agree on the event $\mathcal E$, it follows that
	\[
		\mathcal E \subset \{T_k<\infty; (Z_k\backslash Z_{\ell+1})\subset A\} \subset \{|(Z_k\backslash Z_\ell)\cap L_0| \leq 1\}.
	\]
	Hence it suffices to find a suitable lower bound for the probability of $\mathcal E$.

	Notice that $\mathcal E$ is $\tilde{\mathcal B}$ measurable and hence is independent of the $\sigma$-field $\mathcal G_{\ell+2}$ as well as the conditioning $\{Z_{\ell+1}=D;\mathcal X_{T_{\ell+2}}=(v,z);T_{\ell+2}<\infty\}$.
	Thus
	\[
		\P(\mathcal E | \mathcal G_{\ell+2};\{Z_{\ell+1}=D;\mathcal X_{T_{\ell+2}}=(v,z);T_{\ell+2}<\infty\}) = \P(\mathcal E(D,(v,z))).
	\]
	By a union bound have that
	\[
		\P(\mathcal E(D,(v,z)))  \geq \P(T_k<\infty) - \P(Z_k(v,z)\cap A^c \neq \emptyset).
	\]
	Now let $Z'$ be an independent copy of $Z$, then we have just shown that
	\begin{align*}
		&\P(|(Z_{k} \backslash Z_{\ell})\cap L_0|\leq 1 | T_{\ell+2}<\infty;\mathcal G_{\ell+2}) \\
    &\qquad \geq \P(T_k<\infty) - \P(Z'_k(\mathcal X_{T_{\ell+2}})\cap (Z_{\ell+1}\cup L_0) \neq \emptyset|T_{\ell+2}<\infty;\mathcal G_{\ell+2}).
	\end{align*}
	Using the law of total expectation we see that
	\begin{align*}
		\P(|(Z_{k} \backslash Z_{\ell})\cap L_0|\leq 1 | \mathcal L_{\ell+2};\mathcal G_{\ell}) & \geq \E\bigg[\P(|(Z_{k} \backslash Z_{\ell})\cap L_0|\leq 1 | T_{\ell+2}<\infty; \mathcal G_{\ell+2})\1_{\mathcal L_{\ell+2}} \big| \mathcal G_{\ell}\bigg] \\
		& \geq \left(\P(T_k<\infty) - \P(Z'_k(\mathcal X_{T_{\ell+2}})\cap (Z_{\ell+1}\cup L_0) \neq \emptyset|\mathcal L_{\ell+2};\mathcal G_{\ell})\right)\P(\mathcal L_{\ell+2}|\mathcal G_\ell)
	\end{align*}
	and hence it suffices to show the claimed upper bound for $\P(Z'_k(\mathcal X_{T_{\ell+2}})\cap (Z_{\ell+1}\cup L_0)=\emptyset|\mathcal L_{\ell+2};\mathcal G_\ell)$.

	For convenience let $\mathbf P(\cdot):=\P(\cdot|\mathcal L_{\ell+2};\mathcal G_\ell)$, then
	\begin{multline}\label{eq:decomposition}
		\mathbf{P}\left(Z'_k(\mathcal{X}_{T_{\ell+2}})\cap (Z_{\ell+1}\cup L_0)\neq \emptyset\right) \\
		\leq \mathbf{P}\left(Z'_k(\mathcal{X}_{T_{\ell+2}})\cap (Z_{\ell}\cup L_0) \neq \emptyset\right) +\mathbf{P}\left(Z'_k(\mathcal{X}_{T_{\ell+2}})\cap \{X_{\ell+1}\} \neq \emptyset\right).
	\end{multline}
	For the first term, using Markov's inequality and the fact that $Z'_k(v,z)$ has the same distribution as $v+Z'_k(0,0)$, which we shortcut to $v+Z'_k$,
	\begin{align}\label{eq:intersectionSizeCalculations}
		\mathbf{P}\left(Z'_k(\mathcal{X}_{T_{\ell+2}})\cap (Z_{\ell}\cup L_0) \neq \emptyset\right) & \leq \mathbf{E}\left(|Z'_k(\mathcal{X}_{T_{\ell+2}})\cap (Z_{\ell}\cup L_0)|\right) \\
		& =\sum_{v\in V} \mathbf{P}({X}_{{\ell+2}}=v )\mathbf{E}\left(|v+Z'_k\cap (Z_{\ell}\cup L_0)|\right) \nonumber\\
		& \leq \max_{v\in V} \mathbf{P}({X}_{{\ell+2}}=v) \sum_{v\in V} \mathbf{E}\left(|v+Z'_k\cap (Z_{\ell}\cup L_0)|\right).\nonumber
	\end{align}
	Lemma~\ref{lem:combinatorial} gives that
	\begin{equation*}
		\mathbf{E}\left(\sum_{v\in V}|(v+Z'_k)\cap (Z_{\ell}\cup L_0)|\right) = \mathbf{E}\left(|Z'_k||(Z_{\ell}\cup L_0)|\right)\leq k( \ell +n)
	\end{equation*}
	and hence
	\[
		\mathbf{P}\left(Z'_k(\mathcal{X}_{T_{\ell+2}})\cap (Z_{\ell}\cup L_0) \neq \emptyset\right) \leq k(\ell+n)\max_{v\in V} \mathbf{P}( {X}_{\ell+2}=v).
	\]

	Now we bound the second term of~\eqref{eq:decomposition}.
	We distinguish two cases depending on if $X_{\ell+1}$ and $X_{\ell+2}$ are in the same row/column or not.

	For the first case we assume without loss of generality that $X_{\ell+1}$ and $X_{\ell+2}$ are in the same row $L_i$.
	In this case for any $v \in L_i\backslash \{(0,0)\}$,
	\[
		\mathbf{P}\left(Z'_k(\mathcal{X}_{T_{\ell+2}})\cap \{X_{\ell+1}\} \neq \emptyset\right) = \P(Z_k(v,0)\cap \{(0,0)\}\neq \emptyset),
	\]
	Applying Lemma~\ref{lem:combinatorial} we obtain
	\begin{align*}
		\mathbf{P}\left(Z'_k(\mathcal{X}_{T_{\ell+2}})\cap \{X_{\ell+1}\}=\emptyset\right) & = \frac{1}{n-1} \sum_{v\in L_0\setminus \{(0,0)\}} \E[|Z_k(v,0)\cap \{(0,0)\}|] \\
		& \leq \frac{1}{n-1}  \E\left[\sum_{v\in L_0} |(v+Z_k)\cap \{(0,0)\}|\right]\\
		& = \frac{1}{n-1} \E[|Z_k \cap L_0|]\\
		& \leq \frac{k}{n-1}.
	\end{align*}
The second case follows similarly to \eqref{eq:intersectionSizeCalculations} and is skipped. The proof is thus concluded.
\end{proof}

The quality of the estimate provided by Lemma~\ref{lemma:one_disjoint_line} depends on the choice of the event $\mathcal{L}_{\ell+2}$.
We will set the event $\mathcal L_{\ell+2}$ to be the aforementioned L shaped jump, that is, we want it to be the event that $X_{\ell}$ and $X_{\ell+1}$ are on the same column, and $X_{\ell+1}$ and $X_{\ell+2}$ are on the same row.
Additionally we require that $\mathcal X$ discovers $X_{\ell+2}$ by jumping directly from $X_{\ell+1}$.
Note that this is not always the case, $\mathcal X$ can discover $X_{\ell+1}$, then retrace some of its steps back to $X_j$ for some $j \leq \ell$ and then jump to $X_{\ell+2}$ from $X_j$.

To be precise, we let
\begin{align*}
  \mathcal L_{\ell+2}:=&\big\{X_{\ell}\text{ and }X_{\ell+1}\text{ are on the same row, and }X_{\ell+1}\text{ and }X_{\ell+2}\text{ are on the same column}\big\}\\
  &\cap \big\{|\mathcal B(X_{\ell},V)| = 2\big\}\cap \big\{|\mathcal B(X_{\ell},X_{\ell+1})| = 1\big\}\\
  &\cap\big\{|\mathcal B(X_{\ell+1},V)| = 2\big\}\cap \big\{|\mathcal B(X_{\ell+1},X_{\ell+2})| = 1\big\}.
\end{align*}

Notice that on the event $\mathcal L_{\ell+2}$, $X_{\ell+1}$ has two bridges.
Further, $\mathcal X$ arrives at $X_{\ell+1}$ from $X_\ell$ using one of the two bridges and $\mathcal X$ arrives at $X_{\ell+2}$ from $X_{\ell+1}$ by crossing the second bridge.

\begin{lemma}\label{lem:LshapedJumpProperties}  Suppose that $\ell \leq n/2$ and $\mathcal L_{\ell+2}$ is defined as above. Then
  \[
    \P(\mathcal L_{\ell+2}|\mathcal G_\ell) \geq \frac{t^2}{4n^2(n-1)^2}e^{-\frac{5t}{n(n-1)}}
  \]
  and
  \[
    \max_{v \in V}\P(X_{\ell+2}=v|\mathcal G_\ell; \mathcal L_{\ell+2}) \leq \frac{4}{n^2}e^{-\frac{3t}{n(n-1)}}.
  \]
\end{lemma}
Before presenting the proof, let us remark that in applying the above lemma it will be the case that $t$ has the same order as $n^2$.
For such $t$, Lemma~\ref{lem:LshapedJumpProperties} gives that as $n\to\infty$, $\P(\mathcal L_{\ell+2}|\mathcal G_\ell)$ is bounded below and
\[
  \max_{v \in V}\P(X_{\ell+2}=v|\mathcal G_\ell; \mathcal L_{\ell+2})= O(n^{-2}).
\]
\begin{proof}
	Throughout we work conditionally on $\mathcal G_\ell$, so that every expression appearing throughout this proof is conditionally on $\mathcal G_\ell$.

  We say that a pair of vertices $(v,w)$ are \emph{eligible} if
  \begin{enum}
    \item $v,w \in Z_\ell^c$ with $v\neq w$,
    \item $v$ is on the same row as $X_\ell$,
    \item $w$ is on the same column as $v$.
  \end{enum}
  Let $\mathcal E$ denote set of pairs of eligible pairs of vertices, then notice that $\P((X_\ell,X_{\ell+1})=(v,w);\mathcal L_{\ell+2})>0$ if and only if $(v,w)\in \mathcal E$.

  Fix $(v,w)\in \mathcal E$, then
  \begin{align*}
    &\big\{(X_{\ell+1},X_{\ell+2})=(v,w)\big\}\cap \mathcal L_{\ell+2} \nonumber\\
    &= \big\{|\mathcal B(X_{\ell},V\backslash \{v\})| = 1\big\}\cap \big\{|\mathcal B(X_{\ell},v)| = 1\big\}\cap\big\{|\mathcal B(v,w)| = 1\big\}\cap \big\{|\mathcal B(v,V\backslash\{w,X_\ell\})| = 0\big\}
  \end{align*}
  and note that all of the events on the right hand side are independent of each other.
  We further break down the last event by writing
  \[
    \big\{|\mathcal B(v,V\backslash\{w,X_\ell\})| = 0\big\} = \big\{|\mathcal B(v,Z^c_\ell\backslash\{w\})| = 0\big\}\cap \big\{|\mathcal B(v,Z_{\ell-1})| = 0\big\}
  \]
  both of which are again independent.

  First, $|\mathcal B(X_{\ell},v)|$, $|\mathcal B(v,w)|$ and $|\mathcal B(v,Z_{\ell}^c\backslash\{w\})|$ are independent Poisson random variables where the first two have means $t/|E|$, $t/|E|$ respectively.
  The mean of $|\mathcal B(v,Z_{\ell}^c\backslash\{w\})|$ depends on how many elements of $Z_\ell$ are on the same row as $v$.
  Nevertheless, this mean is at most $2tn/|E|$ and at least $2t(n-\ell-1)/|E|\geq t(n-2)/|E|$ since $\ell \leq n/2$. Thus
  \[
    \frac{t^2}{|E|^2}e^{-t\frac{n+2}{|E|}}\leq \P(|\mathcal B(X_{\ell},v)| = 1;|\mathcal B(v,w)| = 1;|\mathcal B(v,Z_{\ell}^c\backslash\{w\})| = 0)\leq \frac{t^2}{|E|^2}e^{-2t\frac{n}{|E|}}.
  \]

  Next, $|\mathcal B(v,Z_{\ell-1})|$ is a thinned Poisson random variable and is stochastically dominated by a Poisson random variable with mean $2tn/|E|$, hence
  \[
    \P(|\mathcal B(v,Z_{\ell-1})|=0) \geq e^{-\frac{2tn}{|E|}}.
  \]

  Finally, $\mathcal B(X_{\ell},V\backslash \{v\})$ is a thinned Poisson random variable conditioned to be non-empty (since we know that there is at least one bridge to $X_\ell$ from somewhere).
  The segments of bars that have not been visited by $\mathcal X$ is an independent Poisson process, thus $|\mathcal B(X_{\ell},V\backslash \{v\})|-1$ is stochastically dominated by a Poisson random variable with mean $2t(n-1)/|E|$ and hence
  \[
    \P(|\mathcal B(X_{\ell},Z_\ell^c\backslash \{v\})|=1) \geq e^{-2t\frac{n-1}{|E|}}.
  \]
  Combining the estimates together with the fact that $|E|=n^2(n-1)$, we have that for each $(v,w)\in \mathcal E$,
  \begin{equation}\label{eq:v_w_estimate}
    \frac{t^2}{n^4(n-1)^2}e^{-\frac{5t}{n(n-1)}} \leq \P((X_{\ell+1},X_{\ell+2})=(v,w);\mathcal L_{\ell+2}) \leq \frac{t^2}{n^4(n-1)^2}e^{-\frac{2t}{n(n-1)}}.
  \end{equation}

  Since there are at most $\ell \leq n/2$ many elements of $Z_\ell$ on any row or column, an easy counting argument shows that
  \[
    \frac{n^2}{4} \leq |\mathcal E| \leq n^2.
  \]
  Thus summing over~\eqref{eq:v_w_estimate},
  \begin{equation}\label{eq:L_bound}
    \P(\mathcal L_{\ell+2}) \geq \frac{t^2}{4n^2(n-1)^2}e^{-\frac{5t}{n(n-1)}}
  \end{equation}
  which shows the first claim.

  Next, for $w \in V$, there is at most one $v \in V$ such that $(v,w)$ is eligible. Thus dividing~\eqref{eq:v_w_estimate} by~\eqref{eq:L_bound}, we see that
  \[
    \P(X_{\ell+2}=w|\mathcal L_{\ell+2}) \leq \frac{4}{n^2}e^{-\frac{3t}{n(n-1)}}
  \]
  which shows the second claim.
\end{proof}

Now we combine the previous two lemmas to get a bound on the number of times $Z_k$ intersects the row $L_0$.

\begin{lemma}\label{lemma:intersect_line}
	Let $k, M \in \N$ with $k \leq n/2$, then
	\[
		\P(|Z_{k} \cap L_0| \geq M) \leq \left(1- \frac{t^2}{4n^2(n-1)^2}e^{-4t/n^2}\left(\P(T_k < \infty) - \frac{10k}{n} e^{2t/n^2}\right)\right)^{\lfloor M/2\rfloor}.
	\]
\end{lemma}
\begin{proof}
  Fix $k,M\in \N$ with $k \leq n/2$ and let
  \[
    p:=1- \frac{t^2}{4n^2(n-1)^2}e^{-4t/n^2}\left(\P(T_k < \infty) - \frac{10k}{n} e^{2t/n^2}\right).
  \]
  Then by Lemma~\ref{lemma:one_disjoint_line}, Lemma~\ref{lem:LshapedJumpProperties}  we have that for any $\ell \leq k-2$,
  \begin{equation}\label{eq:combo}
    \P(|Z_k\backslash Z_\ell|>1|\mathcal G_\ell) \leq p.
  \end{equation}

	For a sequence of natural numbers $u_1<\cdots<u_\ell \leq k$, consider the event
	\[
		\mathcal A(u_1,\dots,u_\ell):=\{X_{j}\in L_1, \forall j\in \{u_1,\dots,u_\ell\}; X_{j}\notin L_1, \forall j\in \{1,\dots,u_\ell\}\backslash\{u_1,\dots,u_\ell\}\}.
	\]
	In words, this is the event that the set of times that $X$ intersects $L_1$ at times $u_1,\dots,u_\ell$ before time $u_\ell$.
	In other words, we only observe the first $\ell$ times that $X$ intersects $L_1$.
	Now let
	\[
		\mathcal A_{\ell}:= \bigcup_{u_1<\cdots<u_\ell\leq k}\mathcal A(u_1,\dots,u_\ell).
	\]
	where we note that each term in the union is disjoint.
	The event $\mathcal A_\ell$ is the event that $X$ intersects $L_1$ at least $\ell$ times before time $k$ and hence $\{|Z_k\cap L_1|\geq M\}=\mathcal A_M$.
	We will estimate the probability of $\mathcal A_\ell$ by induction.

	Now notice that if $\{\mathcal A(u_1,\dots,u_{\ell-2}); |(Z_k\backslash Z_{u_{\ell-2}})\cap L_1|\geq 2\}$ holds, then this means that there are is at least two more visits to $L_1$ between the times $u_{\ell-1}$ and $k$.
	Thus
	\begin{align*}
		\P(\mathcal A_\ell) & =\P\left(\bigcup_{u_1<\cdots<u_{\ell-2}\leq k}\mathcal A(u_1,\dots,u_{\ell-2}); |(Z_k\backslash Z_{u_{\ell-2}})\cap L_1|\geq 2 \right)                             \\
		                    & = \sum_{u_1<\cdots<u_{\ell-2}\leq k}\P(\mathcal A(u_1,\dots,u_{\ell-2}); |(Z_k\backslash Z_{u_{\ell-2}})\cap L_1| \geq 2)                                          \\
		                    & = \sum_{u_1<\cdots<u_{\ell-2}\leq k} \E\big[\1_{\mathcal A(u_1,\dots,u_{\ell-2})}\P(|(Z_k\backslash Z_{u_{\ell-2}})\cap L_1|\geq 2 |\mathcal G_{u_{\ell-2}})\big],
	\end{align*}
	where in the second equality we have used disjointness, in the third equality we have used the fact that $\mathcal A(u_1,\dots,u_{\ell-2})$ is $\mathcal G_{u_{\ell-2}}$ measurable.
  Using~\eqref{eq:combo} we get that
  \[
    \P(\mathcal A_\ell) \leq p \sum_{u_1<\cdots<u_{\ell-2}\leq k}\P(\mathcal A(u_1,\dots,u_{\ell-2})) = p \P(\mathcal A_{\ell-2}).
  \]
  The result now follows by induction.
\end{proof}

Roughly speaking the previous lemma states that excursions of the CRW cannot have large intersection with any row or column.
In other words the isoperimetric constant defined in~\eqref{eq:iotaDefinition} is likely to be small.
We we transfer this result to the cycles of the interchange process, which will be useful in other sections.
For $k\in \N$ and $v \in V$ we define
\begin{equation}\label{eq:orbDefinition}
	\orb_t^k(v)=\{\underbrace{\sigma_t\circ\dots\circ\sigma_t}_{\ell}(v): \ell=0,\dots,k\}
\end{equation}
be the first $k$ elements of the cycle of containing $v$.
We write $\orb_t(v)$ for $\orb_t^\infty(v)$.

\begin{lemma}\label{lem:ispoerymetryAtGivenTime}
	Let $k, M\in \N$ set $K=\lceil e^{2t/n^2} k/2\rceil$ and assume that $K \leq n/2$. Then
	\begin{align*}
    \P\left(\max_{v\in V}\iota(\orb^{k}_t(v)) \geq M\right)  \leq &n^4 \left(1- \frac{t^2}{4n^2(n-1)^2}e^{-4t/n^2}\left(\P(T_K < \infty) - \frac{10K}{n} e^{2t/n^2}\right)\right)^{\lfloor M/2\rfloor}\\
    & + n^2 e^{-e^{-t/n^2}k}.
  \end{align*}
\end{lemma}
\begin{proof}
	Fix $v\in V$, $k \in \N$ and set $K=\lceil e^{2t/n^2} k/2\rceil$.
	Consider the CRW started at $(v,0)$ and for $\ell \in \N$, define the event
	\[
		\mathcal A_\ell:=\{|\mathcal B(X_\ell,V)|=1;T_\ell < \infty\} \cup\{T_\ell=\infty\}.
	\]
  As in the proof of Lemma~\ref{lem:LshapedJumpProperties}, we have that continually on $\mathcal G_\ell; T_\ell<\infty$, $|\mathcal B(X_\ell,V)|-1$ is stochastically dominated by a Poisson random variable with mean $2t(n-1)/|E|$. Since $|E|=n^2(n-1)$ we have that
	\[
		\P(|\mathcal B(X_\ell,V)|=1|\mathcal G_\ell; T_\ell<\infty) \geq e^{-2t/n^2}.
	\]
	Hence it follows that
	\begin{equation}\label{eq:orb_dominate_this}
		\P(\mathcal A_\ell| \mathcal G_\ell) \geq e^{-2t/n^2}\P(T_\ell<\infty|\mathcal G_\ell) + \P(T_\ell=\infty|\mathcal G_\ell) \geq e^{-2t/n^2}.
	\end{equation}

	Now if $|\mathcal B(X_\ell,V)|=1$, then it follows that $X_\ell \in \orb^\infty_t(v)$.
	On the other hand if $T_\ell=\infty$ then $\orb_t(v)\subset Z_\ell(v,0)$.
	If $\mathcal A_\ell$ occurs at least $k$ times before time $K$, then it must be the case that $\orb^k_t(v)\subset Z_{K}(v,0)$.
	In other words,
	\[
		\P\left(\orb^k_t(v) \subset Z_{K}\right) \geq \P\left(\sum_{\ell=0}^{K} \1_{\mathcal A_\ell} \geq k\right).
	\]

	Next let $\xi_0,\dots,\xi_{K}$ be a sequence of i.i.d. $\{0,1\}$--valued Bernoulli random variables with parameter $e^{-2t/n^2}$.
	Then from~\eqref{eq:orb_dominate_this}, $\sum_{\ell=0}^{K}\1_{\mathcal A_\ell}$ stochastically dominates $\sum_{\ell=0}^{K}\xi_i$ and hence
	\[
		\P\left(\orb^k_t(v) \subset Z_{K}\right) \geq \P\left(\sum_{\ell=0}^{K}\xi_i > k\right) \geq 1-e^{-e^{-t/n^2}k}
	\]
	where the final inequality follows from Hoeffding's inequality.

	Conditionally on the event $\{\orb^k_t(v) \subset Z_{K}\}$ we have
	\[
		\iota(\orb^k_t(v)) \leq \iota(Z_K).
	\]
	Applying Lemma~\ref{lemma:intersect_line} and a union bound we get that
	\[
		\P(\iota(Z_K) \geq M) \leq n^2 \left(1- \frac{t^2}{4n^2(n-1)^2}e^{-4t/n^2}\left(\P(T_K < \infty) - \frac{10K}{n} e^{2t/n^2}\right)\right)^{\lfloor M/2\rfloor}.
	\]
	This lemma now follows by taking a union bound over $v\in V$.
\end{proof}

Next, we no longer think of $t$ as being fixed and write $T^t_k$ to indicate the dependence of $T_k$ on $t$.
In the final lemma of this section we apply what we have shown so far to obtain an estimate uniform in time.

\begin{lemma}\label{lemma:T_k_to_orb}
  Suppose that there exists a constant $c>0$ such that $t\in[c^{-1}n^2, cn^2]$ and let $\Delta=\Delta(n)\geq 0$ be a sequence with the property that $\Delta(n) \leq n/\log n$. For some $k \leq n/\log n$ suppose that
  \[
    \liminf_{n\to\infty}\inf_{s\in[t-\Delta,t]}\P(T^s_k < \infty) >0.
  \]
  Then there exist two constants $\kappa,C>0$ such that
  \[
    \P\left(\sup_{s\in[t-\Delta,t]}\max_{v\in V}\iota(\orb^{k}_s(v)) \geq \log^2 n\right) \leq C e^{-\kappa\log^2 n}.
  \]
\end{lemma}
\begin{proof}
  Under the assumptions of the lemma we have that by Lemma~\ref{lem:ispoerymetryAtGivenTime} there exists two constant $C,\kappa>0$ such that,
  \begin{equation}\label{eq:sup_outside}
    \sup_{s\in[t-\Delta,t]}\P\left(\max_{v\in V}\iota(\orb^{k}_s(v)) \geq \log^2 n\right) \leq C e^{-\kappa\log^2 n}.
  \end{equation}
  Thus it remains to see how to pull the supremum inside of the probability.

  Set $m:=\lceil \Delta e^{(\kappa/2)\log^2 n}\rceil$ and let $I_1,\dots,I_m$ be any sequence of closed intervals of length $|I_i|\leq e^{-(\kappa/2)/\log^2 n}$ such that $\bigcup_i I_i=[t-\Delta,t]$.
  Then
  \begin{equation}\label{eq:interval_break}
    \P\left(\sup_{s\in[t-\Delta,t]}\max_{v\in V}\iota(\orb^{k}_s(v)) \geq \log^2 n\right) \leq \sum_{i=1}^m \P\left(\sup_{s\in I_i}\max_{v\in V}\iota(\orb^{k}_s(v)) \geq \log^2 n\right).
  \end{equation}
  For each $i \leq m$, let $J_i$ be the event that $\sigma_{t}$ has two or more jumps inside of the interval~$I_i$.
  Since $|I_i|\leq e^{-(\kappa/2)\log^2 n}$, we have that there exists a constant $C'>0$ such that $\P(J_i)\leq C'e^{-\kappa\log^2 n}$.
  Let $a_i=\inf I_i$ and $b_i=\sup I_i$, then on the event $J_i^c$, we have that
  \[
    \sup_{s\in I_i}\max_{v\in V}\iota(\orb^{k}_s(v)) \geq \log^2 n = \max\left\{\max_{v\in V}\iota(\orb^{k}_{a_i}(v)),\max_{v\in V}\iota(\orb^{k}_{b_i}(v)) \right\}.
  \]
  Hence we see that
  \begin{align*}
    \P\left(\sup_{s\in I_i}\max_{v\in V}\iota(\orb^{k}_s(v)) \geq \log^2 n\right)&\leq \P\left(\sup_{s\in I_i}\max_{v\in V}\iota(\orb^{k}_s(v)) \geq \log^2 n ; J_i^c\right) + \P(J_i)\\
    & \leq \P\left(\max_{v\in V}\iota(\orb^{k}_{a_i}(v))\right) + \P\left(\max_{v\in V}\iota(\orb^{k}_{b_i}(v))\right) + C'e^{-\kappa \log^2 n}\\
    & \leq C'' e^{-\kappa\log^2 n}
  \end{align*}
  for some constants $C',C''>0$, where in the final inequality we have used~\eqref{eq:sup_outside}.
  Plugging this into~\eqref{eq:interval_break} we see that
  \[
    \P\left(\sup_{s\in[t-\Delta,t]}\max_{v\in V}\iota(\orb^{k}_t(v))\geq \log^2 n\right) \leq C'' m e^{-\kappa\log^2 n}.
  \]
  The result now follows from the fact that $m=O(e^{(\kappa/2)\log^2 n})$.
\end{proof}

\section{Cycle lengths under good isoperimetry}\label{sec:longCyclesUnderGoodIsoperymetry}

In this section we prove results about the cycle sizes when we assume that we have good isoperimetry.
The arguments are based on adaptations of the techniques in~\cite{MR2754801}.
There are two crucial differences which we make here.
One is that we couple the interchange process with the random graph process at some time $t$, rather than just at $t=0$.
The advantage of this is that the coupling runs for less time which minimises the error terms of the coupling.
The second difference is that we incorporating isopermetry in to the estimates in~\cite{MR2166362}.
This means that our bounds become sharper as we are able to show better bounds on the isoperimetry of the cycles of $\sigma_t$.

Recall \eqref{eq:orbDefinition}, for an interval $I \subset [0,\infty)$ and $k \in \N$ let
\[
	\mathcal I_k(I):=\left\{\sup_{t \in I}\max_{v \in V}\iota(\orb^{k}_t(v)) \leq \log^2 n\right\}
\]
denote the event that any fragment of a cycle of length $k$, started from any vertex, has good isoperimetric properties and further that this holds uniformly for $t \in I$.

We begin with a simple lemma about the probability of splitting into small cycles.

\begin{lemma}\label{lem:intensityOfSplit}  Suppose that for some $k \in \N$,
	\begin{equation}\label{eq:isoperimetricAssumption}
		\max_{v\in V} \iota(\orb_\sigma^k(v)) \leq \log^2 n
	\end{equation}
  and let $e=(v,w)$ be an edge chosen uniformly at random.
  Then for $\ell \geq k$ the probability that a cycle of $\sigma$ is split in $(v,w) \circ \sigma$ into two cycles, one of which has size smaller than $\ell$ is at most
	\[
		 \frac{4\ell}{k n} \log^2n.
	\]
\end{lemma}
\begin{proof}
  For a given vertex $v \in V$, the number of $w \in V$ such that $(v,w)$ is an edge so that a cycle of $\sigma$ is split in $(v,w) \circ \sigma$ into two cycles, one of which has size smaller than $\ell$, is
  \[
     |\orb^{\ell}_\sigma(v)\cup \orb^{-\ell}_\sigma(v)\cap(D\cup L \setminus \{v\})| \leq  2\iota(\orb^{\ell}_\sigma(v)\cup \orb^{-\ell}_\sigma(v)),
  \]
  where $L,D$ are respectively the row and column containing $v$ and $\orb^{-\ell}$ corresponds to composition of $\sigma^{-1}_t$ in \eqref{eq:orbDefinition}. Then by the sub-additivity of $\iota$ we get that,
	\[
		\iota(\orb^{\ell}_\sigma(v)\cup \orb^{-\ell}_\sigma(v)) \leq \left\lceil \frac{2\ell}{k} \right\rceil \max_{v'\in V} \iota(\orb_\sigma^k(v')) \leq \left\lceil \frac{2\ell}{k} \right\rceil \log^2 n.
	\]

	Now, suppose that $e=(v,w)$ is chosen uniformly edge. Considering $v$ fixed, there are $2(n-1)$ many vertices that are neighbouring $v$.
  Thus we see that the probability that a cycle of $\sigma$ is split in $\sigma\circ (v,w)$ into two cycles, one of which has size smaller than $\ell$ is at most
  \[
    \frac{1}{n-1}\left\lceil \frac{2\ell}{k} \right\rceil \log^2 n \leq \frac{4\ell}{kn}\log^2 n.
  \]
\end{proof}

Next we present a coupling between the interchange process and a random graph process.
Let $t \geq 0$ and consider a process $G^{t}=(G^{t}_s:s \geq 0)$ of random graphs on the vertex set $V$, defined as follows.
Initially $G_0^t$ is a graph whose connected components are precisely the cycles of $\sigma_t$.
There may be several graphs that satisfy this and for our purpose it will not matter which one is chosen.
Next, whenever $(\sigma_{t+s}:s \geq 0)$ swaps a pair of particles across an edge $e$, we add $e$ to the graph process. 

Recall that $V_t(\ell)$ is the vertices of $H$ which belong to cycles of length at least $\ell$.
Let $V^{G}_{t,s}(\ell)$ denote the vertices of $G^{t}_s$ which belong to connected components of size at least $\ell$.
One important property of this coupling is that every cycle of $\sigma_{t+s}$ is contained in a connected component $G^{t}_s$.
Hence it follows that $V_{t+s}(\ell)\subset V^G_{t,s}(\ell)$ for every $t,s\geq 0$ and $\ell \in \N$.
We will now estimate $|V^G_{t,s}(\ell)\backslash V_{t+s}(\ell)|$ using a similar argument to that in~\cite[Lemma 2.2]{MR2166362}.

\begin{lemma}\label{lemma:schramm_estimate}
  Let $t,\Delta \geq 0$ and suppose that $\ell,k \in \N$ are such that $k\leq\ell$. Then
  \[
    \E\left[\sup_{s \in [0,\Delta]}|V^G_{t,s}(\ell)\backslash V_{t+s}(\ell)|\right] \leq \frac{4\ell^2\Delta}{k n} \log^2n + \ell\Delta\,\P(\mathcal I_k[t,t+\Delta]^c).
  \]
\end{lemma}

Note that for $k=\log^2 n$, $\P(\mathcal I_k[t,t+\Delta])=1$ and this gives the bound $4 \ell^2 \Delta/n$ which also follows from a straight forward adaptation of~\cite[Lemma 2.2]{MR2166362}. Starting from this bound, we will later see that the term $\ell\Delta\,\P(\mathcal I_k[t,t+\Delta]^c)$ becomes negligible for any $k = o(n)$. In this way we obtain a much better bound, which is crucial for proving our result.

\begin{proof}
  Let $I$ be the set of $s \in [0,\Delta]$ such that $\sigma$ experiences a fragmentation at time $t+s$ which splits a cycle and at least one of the resulting cycles has length less than $\ell$.
  From Lemma~\ref{lem:intensityOfSplit} we obtain that at time $u$ the rate of fragmentations where one piece is smaller than $\ell$ is at most
	\[
		\frac{4\ell }{k n}\log^2n + \1_{\{\mathcal I_k(\{u\})^c\}}.
	\]
  Hence we see that
  \begin{equation}\label{eq:I_schramm}
    \E[|I|]\leq \frac{4\ell\Delta}{k n} \log^2n + \Delta\, \P(\mathcal I_k[t,t+\Delta]^c).
  \end{equation}

  Let $s \in [0,\Delta]$ and suppose that $v \in V^G_{t,s}(\ell)\backslash V_{t+s}(\ell)$.
  Then it follows that the cycle containing $v$ must have fragmented between at some time $u \in [t,t+s]$ producing a cycle of size smaller than $\ell$.

  Consider the maximal time $u \in [t,t+s]$ that the cycle containing $v$ fragments.
  Then at this time $u$, $\sigma$ experiences a fragmentation which splits a cycle into two and at least one of the resulting cycles has length less than $\ell$.
  Hence it follows that $u\in I$ and $|V^G_{t,s}(\ell)\backslash V_{t+s}(\ell)| \leq \ell |I|$.
  Taking supremums and using~\eqref{eq:I_schramm} we obtain the desired result.
\end{proof}

In our applications, we will often know that for certain $\ell\in N$, $V^G_{t,0}(\ell) \geq Cn^2$ for some constant $C>0$.
Given this, we wish to know how long it takes until we see components of size comparable to $n^2$.
We do this in the next lemma by using a sprinkling argument first introduced by \textcite{MR671140}.

\begin{lemma}\label{lemma:sprinkle}
  Let $t \geq 0$ and $\ell\in \N$ such that $\ell\leq n^2$. Then for any constant $\delta\in(0,1)$ and any $s \geq (n^2/\ell)\log n$
  \[
    \lim_{n\to\infty}\P\left(|V^G_{t,s}(\delta n^2/8)| \geq \delta n^2/8 \big| |V^G_{t,0}(\ell)|\geq \delta n^2\right)=1.
  \]
\end{lemma}
\begin{proof}
  Throughout we will work conditionally on the event $V^G_{t,0}(\ell) \geq \delta n^2$.
  Since $V^G_{t,s}(\ell)\subset V^G_{t,s'}(\ell)$ for $s\leq s'$, it suffices to consider only the case when $s = (n^2/\ell)\log n$.

  The event $\{V^G_{t,s}(\delta n^2/8) < \delta n^2/8\}$ implies that $V^G_{t,0}$ can be partitioned into two sets $A$ and $B$, each of size at least $\delta n^2/4$, such that the vertices in $A$ and $B$ are not connected in $G^t_s$.
  Now consider two fixed sets $A$ and $B$ which partition $V^G_{t,0}$ and let $\mathcal C(A,B)$ be the event that the vertices in $A$ and $B$ are not connected in $G^{t}_s$.

  Let
  \[
    D:=\{v \in V: E(\{v\},A)\geq \delta^2n/64 \text{ and }E(\{v\},B)\geq \delta^2n/64\}
  \]
  be the set of vertices that have at least $\delta^2n/64$ many neighbours in both $A$ and $B$. Notice that
  \begin{align}\label{eq:C_bound}
    \P(\mathcal C(A,B)) &\leq \prod_{v \in D}\P(v\text{ is not connected to }A\text{ or }B)\nonumber\\
    &=  \left(1-\left(1-e^{-s\frac{\delta^2 n}{64 |E|}}\right)^2\right)^{|D|}.
  \end{align}
  Now we bound $|D|$.
  Notice that there at least $\delta^2 n^4 /16$ many paths of length $2$ between $A$ and $B$.
  On the other hand, for every $v \notin D$, there are at most $\delta^2 n^2/32$ many paths of length $2$ between $A$ and $B$ with $v$ as the mid-point.
  Every $v \in D$ can create at most $ 4 n^2$ many paths of length $2$ between $A$ and $B$ with $v$ as the mid-point.
  Hence the total number of paths of length $2$ between $A$ and $B$ is bounded from above by $(\delta^2 n^2/32)|D^c|+4n^2|D|$.
  Combining with the lower bound we get that
  \[
    \frac{\delta^2 n^2}{32}(n^2 - |D|)+4n^2|D| \geq \frac{\delta^2 n^4}{16}
  \]
  and thus there exists a constant $\rho>0$ such that $|D|\geq \rho n^2$.
  Plugging this into~\eqref{eq:C_bound} we see that
  \[
    \P(\mathcal C(A,B)) \leq \left(1-\left(1-e^{-s\frac{\delta^2 n}{64 |E|}}\right)^2\right)^{\rho n^2}.
  \]
 Next notice that there are at most $2^{n^2/\ell}$ many partitions $A$ and $B$ of the set $V^G_{t,0}(\ell)$, hence
  \[
    \P(V^G_{t,s}(\delta n^2/8) < \delta n^2/8) \leq 2^{\frac{n^2}{\ell}}\left(1-\left(1-e^{-s\frac{\delta^2 n}{64 |E|}}\right)^2\right)^{\rho n^2}.
  \]
  Since $s=(n^2/\ell)\log n$ and $|E|=n^2(n-1)$, we have that there exist constants $C,c>0$ such that
  \[
    2^{\frac{n^2}{\ell}}\left(1-\left(1-e^{-s\frac{\delta^2 n}{16 |E|}}\right)^2\right)^{\rho n^2} \leq C \exp\left\{\frac{n^2}{\ell}\log 2 - c\frac{n^2\log n}{\ell}\right\}.
  \]
  Since $\ell\leq n^2$ and $n \to\infty$, we have that the above converges to zero.
\end{proof}

Finally we combine the last two lemmas with Lemma~\ref{lemma:T_k_to_orb} to show a result that will allow us recursively to obtain better bounds on the lengths of cycles.

\begin{lemma}\label{lemma:recursion}
  Suppose that there exists a constant $c>0$ such that $t \in[c^{-1}n^2, cn^2]$. Let $\ell = \ell(n)$ and $\Delta=(n^2/\ell)\log n$ and $k=\min\{\ell,n/\log n\}$ and $\delta>0$ be such that
\[
	\lim_{n\to+\infty} \inf_{s\in[t-2\Delta,t]} \P(V_{s}(\ell)>\delta n^2)=1.
\]
 Then there exists a $\rho>0$ such that
  \[
    \lim_{n\to\infty}\P\left(\inf_{s\in [t-\Delta,t]}\left|V_s\left(\frac{\sqrt{\ell k n}}{\log^2 n}\right)\right|>\rho n^2\right)=1.
  \]
\end{lemma}

\begin{proof}
  By Lemma~\ref{lemma:sprinkle} we have that
  \begin{equation}\label{eq:giant_comp}
    \lim_{n\to\infty}\P\left(\inf_{s\in[t-\Delta,t]}|V^G_{t,s}(\delta n^2/8)| \geq \delta n^2/8 \right)=1
  \end{equation}
  where we have also used the fact that $V^G_{t,s}(\ell)\subset V^G_{t,s'}(\ell)$ whenever $s\leq s'$.

  Next set $k=\min\{\ell,n/\log n\}$. As $\{T_\ell<\infty\}\subset\{T_k<\infty\}$ we have that by the vertex transitivity of the graph for any $s\in [t-\Delta, t]$ we have
  \[
  	\P(T^{s}_k<\infty) \geq\P(T^{s}_\ell<\infty) = \frac{1}{|V|}\sum_{v\in V} \P(T^{s}_\ell(v,0)<\infty) \geq \frac{1}{n^2}\E(V_s(\ell)).
  \]
  Thus by the assumption
  \[
    \liminf_{n\to\infty}\inf_{s\in [t-2\Delta,t]}\P(T^{s}_k<\infty) \geq \delta.
  \]
By Lemma~\ref{lemma:T_k_to_orb}, there exist two constants $C,\kappa>0$ such that
  \[
    \P(\mathcal I_k[t-2\Delta,t]^c)=\P\left(\sup_{s\in [t-2\Delta,t]}\max_{v\in V}\iota(\orb^{k}_s(v)) \geq \log^2 n\right)\leq Ce^{-\kappa \log^2 n}
  \]
  Thus by Lemma~\ref{lemma:schramm_estimate},
  \begin{align*}
    \E&\left[\sup_{s \in [0,2\Delta]}|V^G_{t-2\Delta,s}(\sqrt{\ell k n}/\log^2 n)\backslash V_{t-2\Delta+s}(\sqrt{\ell k n}/\log^2 n)|\right] \\
    &\leq \frac{2(\sqrt{\ell k n}/\log^2 n)^2\Delta}{k n} \log^2n + \ell \Delta\,\P(\mathcal I_k[t,t+\Delta]^c) \\
    &= O\left(\frac{n^2}{\log n}\right).
  \end{align*}
  Using the above, together with~\eqref{eq:giant_comp} and Markov's inequality gives the desired result.
\end{proof}

\section{Proof of Theorem~\ref{thm:phase}}\label{sec:mainProof}

\subsection{Subcrtical phase}
We recall that $(G^0_s, s\geq 0)$ is a random graph process in which a uniformly chosen edge is added at rate $1$.
Let $\beta<1/2$ and set $t=\beta n^2$.
Then by using standard branching arguments (see for example the proof of~\cite[Theorem 2.3.1]{MR2656427}) it is possible to show that there exists $C>0$ such that
\[
  \lim_{n\to\infty} \P(|V^G_{0,t}(C\log n)|=0) = 1.
\]
Now the subcritical phase of Theorem~\ref{thm:phase} is a consequence of the fact that under our coupling $V_{t}(C\log n)\subset V^G_{0,t}(C\log n)$.

\subsection{Supercritical phase}

Throughout we let $\beta>1/2$ and set $t=\beta n^2$. We will proceed inductively applying Lemma~\ref{lemma:recursion} repeatedly to obtain larger and larger cycles. Formally, fix sequence of positive real numbers $(\alpha_h,h \in \mathbb{N})$ such that:
\begin{enumerate}
	\item $\alpha_1 \in (0,1/2)$,
	\item for any $h\geq 1$, $\alpha_{h+1} < (1+\alpha_h +\min\{1,\alpha_h\})$,
	\item $\lim_{h\to \infty} \alpha_h = 2$.
\end{enumerate}

We will show by induction that there exists sequence of positive real numbers $(\rho_h,h\in\mathbb{N})$ for which
\begin{equation}\label{eq:induction}
  \lim_{n\to\infty}  \P\left(\inf_{s\in [t-\Delta_h,t]}\left|V_s(n^{\alpha_h})\right|>\rho_h n^2\right)=1,
\end{equation}
where $\Delta_h = n^{2-\alpha_h} \log n$, which implies the statement in Theorem~\ref{thm:phase} in the supercritical phase.

We begin showing the base case $h=1$ by estimating the sizes of the connected components of $G^{0}_s$.
\begin{lemma}\label{lemma:percolation} Let $\beta>1/2$ and set $t=\beta n^2$. Then there exists a constant $\delta>0$ such that
  \[
    \lim_{n\to\infty}\P\left(\inf_{s\in [t-\Delta_1,t]}|V_{0,s}(\delta n^2)|\geq \delta n^2\right)=1.
  \]
\end{lemma}
\begin{proof}
	Fix $\beta>1/2$ and $\beta' \in (1/2, \beta)$ and set $t=\beta n^2$ and $s=\beta' n^2$. Using standard branching techniques (see for example~\cite[Theorem 2.3.2]{MR2656427} and \cite[Lemma 2.3.4]{MR2656427}) one can show that there exists a $\gamma>0$ such that
  \[
    \lim_{n\to\infty}\P\left(|V_{0,s}(\log^2 n)|\geq \gamma n^2\right)=1.
  \]
  The result now follows from Lemma~\ref{lemma:sprinkle}.
\end{proof}
%
%
%

Next we note that for $k=\log^2 n$, we have that trivially $\mathcal I_k[0,\infty)$ holds.
Thus applying Lemma~\ref{lemma:schramm_estimate} we see that and noting that $\alpha_1 <1/2$,
\[
  \E\left[\sup_{s \in [0,t]}|V^G_{0,s}( n^{\alpha_1})\backslash V_{s}(n^{\alpha_1})|\right] \leq \frac{4n^{2\alpha_1} t}{n}=o(n^2).
\]
Now an application of Markov's inequality and using the above together with Lemma~\ref{lemma:percolation} gives that there exists a $\rho_1>0$ such that
\[
\lim_{n\to\infty}  \P\left(\inf_{s\in [t-\Delta_1,t]}\left|V_s(n^{\alpha_1})\right|>\rho_h n^2\right)=1
\]
which shows the case $h=1$ in~\eqref{eq:induction}.

Now we induct. Suppose that \eqref{eq:induction} holds for some $h\geq 1$. We apply Lemma \ref{lemma:recursion} with $\ell(n) = n^{\alpha_h}$ and $\delta = \rho_h$ to get that there exists a $\rho_{h+1}>0$ such that
\[
  \lim_{n\to\infty}  \P\left(\inf_{s\in [t-\Delta_{h+1},t]}\left|V_s(n^{\alpha_{h+1}})\right|>\rho_{h+1} n^2\right)=1
\]
which finishes the proof.
%
%
%

\section{Open questions and discussion}\label{sec:openQuestionsAndDiscussion}

The results in this paper imply that for $t=\beta n^2$ with $\beta>1/2$, for any $\ell=\ell(n)$, there exists two constants $C,\kappa>0$ such that
\begin{equation}\label{eq:almostGoodIota}
	\P\left(\max_{v\in V}\iota(\orb^{n^{\ell}}_{t}(v)) \geq \frac{\ell}{n}\log^3 n\right) \leq c e^{-\kappa\log^2 n}.
\end{equation}
We conjecture that sharper bounds can be obtained, namely, there exists a constant $C>0$ such that for any $\ell\geq n \log^2 n$
\begin{equation}\label{eq:goodIota}
	\P\left(\max_{v\in V}\iota(\orb^{\ell}_t(v)) \geq C \ell /n \right) \leq f(n),
\end{equation}
for some fast decaying function $f(n)$.
In other words~\eqref{eq:almostGoodIota} is sharp up to logarithmic terms.
We believe that showing~\eqref{eq:goodIota} and adapting the techniques of \cite{MR2754801}, one can obtain the existence of cycles of macroscopic length during some small time interval around $t=\beta n^2$ for any $\beta>1/2$.
Let us describe some of the difficulties that arise in proving~\eqref{eq:goodIota}.
The result of Lemma~\ref{lemma:intersect_line} is sharp and can be used to show that for any $i$ and $v$ the tail of $X_i(v) = |\orb^{n}_t(v)\cap L_i|$ decays exponentially.
In order to deduce \eqref{eq:almostGoodIota} we use a union bound over $i$ and $v$, which results in logarithmic errors.
One way to avoid this would be to prove with high probability $|\orb^{n \log^2n }_t(v)\cap L_i|\leq C \log^2n$.
To obtain such a result we would need to understand what happens when the CRW intersects its past (at the moment we avoid this difficulty by assuming $k\leq n/2$ in Lemma~\ref{lemma:intersect_line}).
Another way, would be to find a notion of isoperimetry different than $\iota$. Essentially one needs to be able show that for some constant $C>0$, with high probability $\max_v\sum_{i=0}^{n-1}X_i(v)^2 \leq C n$ (using Lemma \ref{lem:isoperymetryByIota} and \eqref{eq:almostGoodIota} we only obtain the bound $n \log^3n $).
However, it is far from clear how to handle the dependence of the random variables $X_i(v)$.

%

We believe the proof in this paper applies to other Hamming graphs. For $d,k\in \N$, the $(d,k)$-Hamming graph $H(d,k)$ is a graph on the vertices $\{0,\dots,k-1\}^d$ where an edge is present between any two vertices which differ on exactly one co-ordinate. We believe that the proof in this paper should easily adapt to every $H(d,n)$ where $d$ is fixed.

It would be very interesting to investigate the case when $d$ is varying. Arguably, the most interesting case is $H(n,2)$ which is the hypercube. We make the following conjecture.

\begin{conjecture}
  Consider the interchange process $\sigma=(\sigma_t:t \geq 0)$ on $H(d,k)$ where at least one (or possibly both) of $d,k$ is increasing with $n$. Let $\beta>1$ and set $t=\beta d k^d/(2d-2)$. Then there exists a constant $C>0$ such that
  \[
    \lim_{\delta\downarrow 0}\lim_{n\to\infty}\P\big(|V_t(\delta k^d)| \geq C k^d\big)=1.
  \]
\end{conjecture}

We also believe that many of the results that hold on the complete graph also hold on $2$-dimensional Hamming graph.
For example, \textcite{MR2166362} shows that in the supercritical phase on the complete graph, the cycle lengths suitably rescaled converge in distribution to a Poisson-Dirichlet random variable.

The interchange process is related to the loop representation of the correlation on the $1/2$ quantum Heisenberg ferromagnet. The model in question is defined by the measure $\mathbb{Q}$ given by
	\[
		\mathbb{Q}(\sigma_t = \sigma) =\frac{1}{Z} 2^{\#\text{cycles}(\sigma)} \P(\sigma_t = \sigma)
	\]
  with some normalising constant $Z>0$.
  Recently, \textcite{ECP4328} proved existence of macroscopic cycles on the complete graph for the model with any weight $\theta>1$ in place of $2$. It would be interesting to see if it is possible to combine our proof with his methods to obtain an analog of Theorem~\ref{thm:phase} for the measure $\mathbb{Q}$.

\subsection*{Acknowledgments} We would like to thank Nathana\"el Berestycki, Roman Koteck\'y and Daniel Ueltschi for useful discussions.
The research of PM was supported by the National Science Centre UMO-2014/15/B/ST1/02165 grant.
B\c{S} acknowledges support from EPSRC grant number EP/L002442/1.
\renewcommand*{\bibfont}{\small}
\printbibliography
\end{document}